\documentclass[12pt]{amsart}

\usepackage{amstext}
\usepackage{amsthm}
\usepackage{amsmath}
\usepackage{latexsym}
\usepackage{amsfonts}
\usepackage{mathtools}
\usepackage[dvipsnames]{xcolor}
\usepackage{graphicx}
\newtheorem*{theorem*}{Theorem}
\newtheorem*{proposition*}{Proposition}
\newtheorem*{lemma*}{Lemma}
\newtheorem*{conjecture*}{Conjecture}
\newtheorem*{open problem*}{Open Problem}
\newtheorem*{problem*}{Problem}

\newtheorem{theorem}{Theorem}

\newtheorem{proposition}[theorem]{Proposition}

\theoremstyle{definition}\newtheorem{remark}[theorem]{Remark}
\theoremstyle{definition}

\allowdisplaybreaks

\newcommand\CC{{\mathbb C}}
\newcommand\RR{{\mathbb R}}
\newcommand\DD{{\mathbb D}}
\newcommand\TT{{\mathbb T}}
\newcommand\DDD{{\mathcal D}}
\newcommand\BB{{\mathbb B_n}}
\newcommand\SSS{{\mathbb S_n}}

\title[Non-cyclicity and polynomials]{Non-cyclicity and polynomials in Dirichlet-type spaces of the unit ball}
\author{Dimitrios Vavitsas}
\email{dimitris.vavitsas@doctoral.uj.edu.pl}
\author{Konstantinos Zarvalis}
\address{Department of Mathematics, Aristotle University of Thessaloniki, 54124, Thessaloniki, Greece}
\email{zarkonath@math.auth.gr}
\date{}
\thanks{Partially supported by NCN grant SONATA BIS no.  2017/26/E/ST1/00723 of the National Science Centre, Poland}

\begin{document}
\begin{abstract}
We give a description of the intersection of the zero set with the unit sphere of a zero-free polynomial in the unit ball of $\CC^n.$ This description leads to the formulation of a conjecture regarding the characterization of polynomials that are cyclic in Dirichlet-type spaces in the unit ball of $\CC^n$. Furthermore, we answer partially ascertaining whether an arbitrary polynomial is not cyclic. 
\end{abstract}

\maketitle

\section{Introduction}

Spaces of holomorphic functions either in the complex plane $\CC$ or in the setting of several complex variables have been studied extensively with regard to their various properties. One of the more standard procedures in this field is the investigation of cyclic for the shift operator vectors belonging in such spaces. To be precise, suppose that $\mathcal{H}$ is a space of holomorphic functions in a subset of $\CC^n$, $n\ge1$. A function $\phi\in\mathcal{H}$ is called a \textit{multiplier} of $\mathcal{H}$ if $\phi f\in\mathcal{H}$ for every $f\in\mathcal{H}$. Let $\mathcal{M}(\mathcal{H})$ denote the set of all multipliers of $\mathcal{H}$. Given $f\in\mathcal{H}$, we may construct its \textit{closed invariant subspace} $[f]_\mathcal{H}$ which is the closure of the set $\{\phi f:\phi\in\mathcal{M}(\mathcal{H})\}$. When $[f]_\mathcal{H}$ equals the whole space $\mathcal{H}$, we say that $f$ is \textit{cyclic} in $\mathcal{H}$.

The main problem in this field is providing a characterization of cyclicity. When restricting to spaces of holomorphic functions in one complex variable, there are remarkable results which provide the desired characterizations. For example, in the classical \textit{Hardy space} $H^2$ of the unit disk $\DD=\{z\in\CC:|z|<1\}$, it has been proved in \cite{Beurling} that a function $f\in H^2$ is cyclic if and only if it is \textit{outer}, meaning
$$\log|f(0)|=\frac{1}{2\pi}\int\limits_{0}^{2\pi}\log|f(e^{i\theta})|d\theta \quad\text{and}\quad f(0)\ne0.$$
For other well-known spaces such as the \textit{Bergman space}
$A^2$ and the \textit{Dirichlet space} $\DDD$ there exist some partial results, but not full characterizations. For a profound presentation of the rich theory of such spaces and various results concerning cyclicity, we refer the interested reader to \cite{Duren-Schuster} and \cite{Primer}.

Given that the situation remains mostly unclear in $\CC$, one can guess that even less is known in several complex variables. When dealing with spaces of holomorphic functions in $\CC^{n}$, the two principal reference domains are the \textit{polydisk} $\DD^n=\{(z_1,\dots,z_n)\in\CC^{n}:|z_1|<1,\dots,|z_n|<1\}$ and the \textit{unit ball} $\BB=\{(z_1,\dots,z_n)\in\CC^n:|z_1|^2+\dots+|z_n|^2=1\}$. The majority of the results concern the \textit{bidisk} $\DD^2$. Indeed, in \cite{Guo} the authors work with the Hardy space of the bidisk, while in \cite{Kosinski-Sola1} and \cite{Kosinski-Sola2}, the authors discuss the cyclicity of polynomials in Dirichlet-type spaces of the bidisk $\DD^2$. Moreover, results about the cyclicity of polynomials in Dirichlet-type spaces of the unit ball of $\CC^2$ were found in \cite{Vavitsas1}. In the setting of $\CC^n$, $n\ge3$, partial results about the Hardy space of the polydisk $\DD^n$ are obtained in \cite{Linus}, whereas non-cyclicity and cyclicity for special polynomials in the unit ball of $\CC^n$ is examined in \cite{Vavitsas2}. Optimal approximants of $1/f$ and connections with orthogonal polynomials in $\CC^n$ in certain weighted spaces are discussed in \cite{Sargent}, \cite{Sargent2}. Furthermore, recent advances about the case of the Drury-Arveson space $H_d^2$ may be found in \cite{Perfekt}, \cite{Chalmoukis} and \cite{Richter} (see \cite{Hartz} for more information on the Drury-Arveson space).

At this point, let us also note that the theory of spaces of holomorphic functions in the polydisk is quite different compared to the one in the unit ball. This is due to the topology of the two sets; they are not biholomorphic. For instance, the fixed parameters where an abritary polynomial is cyclic or not in the bidisk setting slightly differ from the two-dimentional ball due to the Shilov boundary of each domain; in particular $\textup{dim}_{\RR}(\TT^2)<\textup{dim}_{\RR}(\mathbb{S}_2)=\textup{dim}_{\RR}(\partial \DD^2)$. Nevertheless, at the same time, the two theories present similarities in terms of certain tools which may be utilized in both settings. In addition, on top of these two standard domains, one might further inquire whether it is possible to extend well known results to more general domains such as the pseudoconvex Reinchardt domains containing the origin.

\subsection{Dirichlet-type spaces}
Our objective is to give a sufficient condition for the non-cyclicity of polynomials in the Dirichlet-type spaces of the unit ball $\BB$ of $\CC^n$. Before getting to the formal statement of our results, we are going to need some terminology. For $z=(z_1,\dots,z_n)$ and $w=(w_1,\dots,w_n)\in\CC^n$ we denote by $\langle z,w\rangle=z_1\bar{w_1}+\dots+z_n\bar{w_n}$ the usual \textit{Euclidean inner product}. We use the notation $||z||=\sqrt{|z_1|^2+\dots+|z_n|^2}$ for the associated induced norm. So, for the unit ball we have $\BB=\{z\in\CC^n:||z||<1\}$ and for its boundary, the \textit{unit sphere}, we have $\SSS=\{z\in\CC^n:||z||=1\}$. Suppose that $f\in\textup{Hol}(\BB)$, where $\textup{Hol}(\BB)$ denotes the space of holomorphic functions in $\BB$. Then $f$ has a power series expansion of the form
$$f(z)=\sum\limits_{|k|=0}^{+\infty}a_kz^k, \quad\quad z=(z_1,\dots,z_n)\in\BB,$$
where $k=(k_1,\dots,k_n)$ is a multi-index of non-negative integers, $|k|=k_1+\dots+k_n$ and $z^k=z_1^{k_1}\cdots z_n^{k_n}$. For a fixed $\alpha\in\RR$, we say that $f$ belongs to the \textit{Dirichlet-type space} $D_\alpha(\BB)$ whenever
$$||f||^2_\alpha:=\sum\limits_{|k|=0}^{+\infty}(n+|k|)^\alpha\frac{(n-1)!k!}{(n-1+|k|)!}|a_k|^2<+\infty,$$
where $k!=k_1!\cdots k_n!$. Special cases of this family are all the classical Hilbert spaces of holomorphic functions in the unit ball of $\CC^n$. Indeed, $D_{-1}(\BB), D_0(\BB), D_{n-1}(\BB)$ and $D_n(\BB)$ coincide with the Bergman, Hardy, Drury-Arveson and Dirichlet spaces, respectively. For more information, we refer the interested reader to \cite{Zhu}.

\subsection{Cyclic vectors}
For $i=1,\dots,n$ consider the \textit{shift operator} $T_i:D_\alpha(\BB)\to D_\alpha(\BB)$ defined by $T_i:f(z)=f(z_1,\dots,z_i,\dots,z_n)\mapsto z_i\cdot f(z)$. A function $f\in D_\alpha(\BB)$ is called a \textit{cyclic vector} if the closed invariant subspace
$$[f]_\alpha:=\textrm{clos span}\{z_1^{k_1}\cdots z_n^{k_n}f:k_1,\dots,k_n=0,1,2,\dots\},$$
where the closure is taken with respect to the norm of $D_\alpha(\BB)$, coincides with the whole space $D_\alpha(\BB)$.

To attack the problem of cyclicity, we are going to study the \textit{zero set} of a polynomial $p\in\CC[z_1,\dots,z_n]$ which will be written from now on as
$$\mathcal{Z}(p):=\{z\in\CC^n:p(z)=0\}.$$
The points lying on $\mathcal{Z}(p)\cap\SSS$ are characterized as the \textit{boundary zeroes} of $p$. Authors concerning cyclic vectors in Dirichlet-type spaces make clear that the cyclicity of a polynomial is intrinsically linked with the real dimension of its zero set. This idea can be demonstrated through the following theorems: \cite[Theorem 4.8]{Perfekt}, \cite[Theorem p.8740]{Kosinski-Sola1}, \cite[Corollary p.289]{Brown-Shields}, \cite[Theorem 1]{Kosinski-Sola2}, \cite[Theorem 3]{Vavitsas1}.

\subsection{Main results}
We will exclusively focus on the polynomials $p\in\CC[z_1,\dots,z_n]$ that are zero-free in $\BB$. This is because if $p$ has a zero inside the unit ball, then it cannot be cyclic for any Dirichlet-type space.

Following the trend of all the aforementioned results, we are going to see that there is an inextricable correlation between a polynomial of $\CC[z_1,\dots,z_n]$ and the nature of its boundary zeroes. Towards this goal, our first task is to provide a geometric description for the boundary zeroes of such polynomials. To do this, we are going to utilize tools stemming from semi-algebraic geometry and function theory in the unit ball.

\begin{theorem}\label{zeroset}
    Let $p\in\CC[z_1,\dots,z_n]$ be a zero-free in $\BB$ polynomial. If $\mathcal{Z}(p)\cap\SSS$ is non-empty, then it is either a finite set, or $\mathcal{Z}(p)\cap\SSS=\cup M_i;$ a finite disjoint union of Nash submanifolds $M_i\subset\RR^{2n}$, where each $M_i$ is Nash diffeomorphic to an open hypercube $(-1,1)^m$, $m=0,1,\dots,n-1$. In particular, each Nash diffeomorphism $\phi_i:(-1,1)^m\rightarrow M_i$ is complex-tangential and real analytic.
\end{theorem}
More information on complex tangential functions, Nash submanifolds and Nash diffeomorphisms follows in the next section.

The situation described in the preceding theorem about the shape of the zero set is actually verified by the corresponding zero set of the so-called \textit{model polynomials}. For $m\in\{1,2,\dots,n\}$, we say that $\pi_m(z_1,\dots,z_n)=1-m^{m/2}z_1\cdots z_m$ are the model polynomials of the Dirichlet-type space $D_\alpha(\BB)$. It can be readily checked that the intersection of their zero sets with the unit ball are the sets
$$\left\{\frac{1}{\sqrt{m}}\left(e^{i\theta_1},\dots,e^{i\theta_{m-1}},e^{-i(\theta_1+\dots+\theta_{m-1})},0,\dots,0\right):\theta_j\in[0,2\pi)\right\}.$$
We get at once that $\textup{dim}_{\RR}(\mathcal{Z}(\pi_m)\cap\SSS)=m-1$ if we consider it as a subset of $\RR^{2n}$. But by \cite{Vavitsas2} we know that $\pi_m$ is cyclic in $D_\alpha(\BB)$ if and only if $\alpha\le\frac{2n-(m-1)}{2}$; the fixed parameters involve real dimensions of the boundary zeroes. Note that the cyclicity of the model polynomials in the two dimensional ball was examined in \cite{Sola}. Based on that work followed the characterization of cyclicity of an arbitrary polynomial in \cite{Vavitsas1}.

\begin{figure}
    \centering
    \includegraphics[scale=0.4]{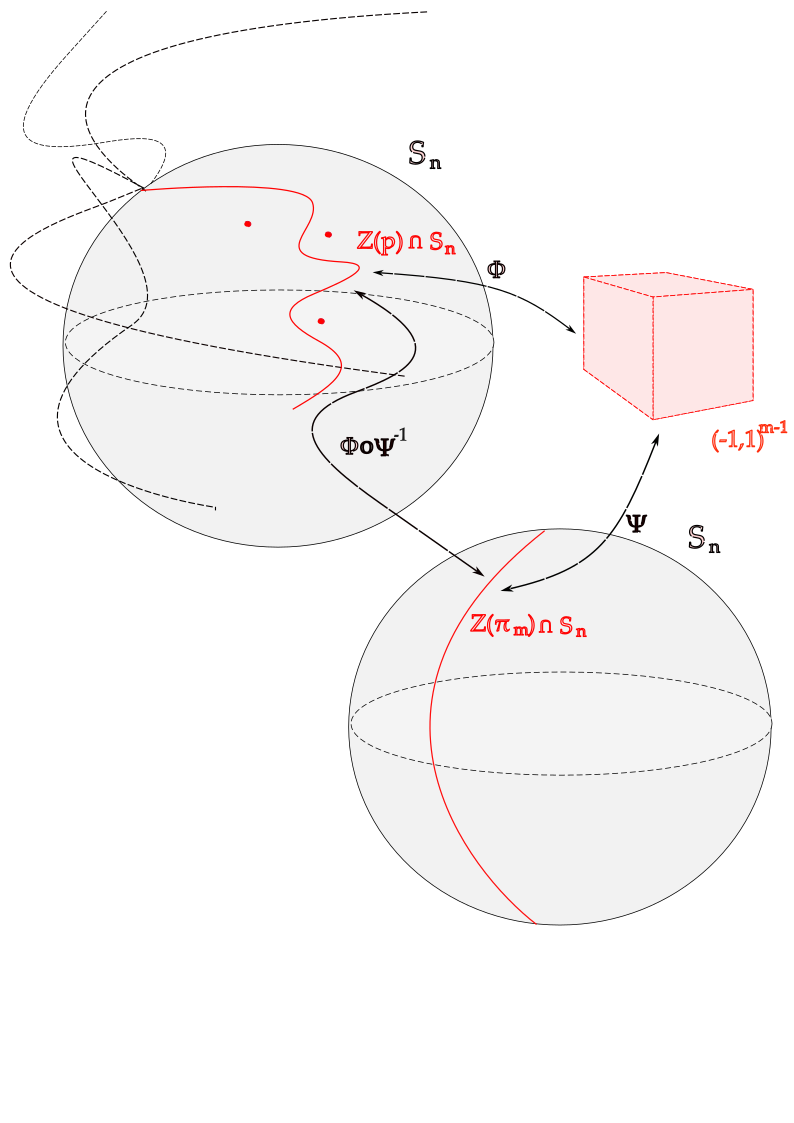}
    \caption{Illustration of the Conjecture}
    \label{figure}
\end{figure}

The combination of these results naturally leads to the following conjecture:

\begin{conjecture*}\label{conjecture}
    Let $p\in\CC[z_1,\dots,z_n]$ be a zero-free in $\BB$ polynomial. Suppose that $\mathcal{Z}(p)\cap\SSS$ contains a real submanifold of $\RR^{2n}$ of dimension $m-1$, $m=2,3,\dots,n$, but no submanifold of any higher dimension. Then $p$ is cyclic in $D_\alpha(\BB)$ if and only if $\alpha\le\frac{2n-(m-1)}{2}$.
\end{conjecture*}

Note that if $\mathcal{Z}(p)\cap\SSS$ consists of finitely many points, then we may argue exactly as in \cite[Section 3]{Vavitsas1}; so we omit these cases from the Conjecture.

Even though an actual illustration of the hypercube and the unit ball of $\CC^n$ is not possible, we can always imagine it as a 3-dimensional object. See Figure \ref{figure} for a depiction of the Conjecture.

The second main result of this present article is a potential theoretic result concerning \textit{Riesz} $\alpha$\textit{-capacity} on the unit sphere $\SSS$ which we will denote by $\textup{cap}_\alpha(\cdot)$ (more information follows in the sequel). This result will provide a sense of invariance for Riesz $\alpha$-capacity under complex tangential Nash diffeomorphisms.

\begin{proposition}\label{proposition}
    Let $M,N\subset\SSS$ be Nash submanifolds such that there exist complex tangential Nash diffeomorphisms $\Phi:(-1,1)^m\to M$, $\Psi:(-1,1)^m\to N$, $m=1,2,\dots,n-1$. Fix $\alpha\in(0,n]$. Then, $\textup{cap}_\alpha(M)>0$ if and only if $\textup{cap}_\alpha(N)>0$.
\end{proposition}

Finally, through Proposition \ref{proposition} we are able to proceed to a partially affirmative answer to the Conjecture.

\begin{theorem}\label{non-cyclicity}
    Let $p\in\CC[z_1,\dots,z_n]$ be a zero-free in $\BB$ polynomial. Suppose that $\mathcal{Z}(p)\cap\SSS$ contains a real submanifold of $\RR^{2n}$ of dimension $m-1$, $m=2,3,\dots,n$, but no submanifold of any higher dimension. Then $p$ is not cyclic in $D_\alpha(\BB)$ whenever $\alpha>\frac{2n-(m-1)}{2}$.
\end{theorem}

The structure of the article is as follows. First, in Section 2 we will prove Theorem \ref{zeroset} providing beforehand all the background from function theory and semi-algebraic geometry that is deemed necessary. Next, in Section 3 we are going to briefly present Riesz $\alpha$-capacity whose notion plays a pivotal role during the proofs of Proposition \ref{proposition} and Theorem \ref{non-cyclicity}.

\section{Description of boundary zeroes}
\subsection{Function theory in the unit ball}
Before delving into the main body of the proof, we need some information and results about function theory in the unit ball. We mostly follow \cite{Rudin}.

Even though we work with polynomials which are defined in the whole $\CC^n$, in general, a function $f\in D_\alpha(\BB)$ may not be well-defined in the unit sphere. Thus, boundary sets require a more delicate approach. Let $K\subset\SSS$ be compact and consider the algebra $A(\BB):=\textup{Hol}(\BB)\cap C(\SSS)$ of all functions holomorphic in the unit ball and continuous on the unit sphere.
\begin{enumerate}
    \item[(i)] $K$ is a \textit{(Z)-set} (\textit{zero set}) for $A(\BB)$ if there exists a function $f\in A(\BB)$ such that $f(\zeta)=0$ on $K$, but $f(z)\ne0$, for all $z\in\overline{\BB}\setminus K$.
    \item[(ii)] $K$ is a \textit{(PI)-set} (\textit{peak-interpolation set}) for $A(\BB)$ if the following property is satisfied: to each $g\in C(K)$ that is not identically zero corresponds some $f\in A(\BB)$ such that $f(\zeta)=g(\zeta)$ on $K$ and $|f(z)|<\max_{\zeta\in K}|g(\zeta)|$, for all $z\in\overline{\BB}\setminus K$.
\end{enumerate}
\begin{remark}\label{equivalence-Z-PI}
By \cite[Theorem 10.1.2]{Rudin} we know that $K$ is a (Z)-set for $A(\BB)$ if and only if it is a (PI)-set for  $A(\BB)$. In particular, both these properties are hereditary. This signifies that if $K$ is a (Z)-set (or (PI)-set) for $A(\BB)$, then every compact subset $F$ of $K$ must also be a (Z)-set (or (PI)-set) for $A(\BB)$.
\end{remark}

Next, we need to recall some definitions regarding complex tangentiality. Let $\Omega\subset\RR^m$, $m\in\mathbb{N}$, be an open set and let $\Phi:\Omega\to\SSS$ be a $C^1$ map. Denote by $J_\Phi(x)$ the Jacobian of $\Phi$ evaluated at some point $x\in\Omega$. Then, we say that $\Phi$ is \textit{complex tangential} if the orthogonality relation 
$$\langle J_\Phi(x)\cdot h,\Phi(x)\rangle=0$$
holds for all $x\in\Omega$ and all $h\in\RR^m$, where $J_\Phi(x)\cdot h$ can be thought of as matrix multiplication. The actual definition of complex tangential functions has to do complex tangent spaces and requires Fr\'{e}chet derivatives for infinite dimensions, but for the purposes of this present work, the aforementioned analytical counterpart will suffice. In the case of $C^1$ curves the definition is more straightforward. Let $I\subset\RR$ be an interval on the real line. Then, a $C^1$ curve $\gamma:I\to\SSS$ is said to be a \textit{complex tangential curve} when
$$\langle\gamma'(t),\gamma(t)\rangle=0, \quad\quad\textup{for all }t\in I.$$
\begin{remark}\label{equivalence-complex-tangential}
    It can be proved (see e.g. \cite[$\S$ 10.5.2]{Rudin} that a $C^1$ map $\Phi:\Omega\to\SSS$, where $\Omega\subset\RR^m$ is open, is complex tangential if and only if for any $C^1$ curve $\gamma:[0,1]\to\Omega$, the map $\Phi\circ\gamma:[0,1]\to\SSS$ is complex tangential.
\end{remark}

In addition, for a $\Phi$ as above, we associate to each $x\in\Omega$ the real vector space
$$V_x:=\{J_\Phi(x)\cdot h: h\in\RR^m\}.$$
We understand that $\Phi$ is complex tangential if for each $x\in\Omega$, $\Phi(x)$ is orthogonal to the associated space $V_x$. Furthermore, we characterize $\Phi$ as \textit{non-singular} if the rank of its Jacobian equals $m$ for every $x\in\Omega$.

Last but not least, we know that for a set $E\subset\RR^n$, its \textit{dimension} may be defined as
$$\textup{dim}(E)=\max\{\textup{dim}(\Gamma):\Gamma\subset E, \Gamma \text{ submanifold of }\RR^n\}.$$
Concerning sets of complex vectors, any $M\subset\CC^n$ may be regarded as a set $E\subset\RR^{2n}$. Through this correspondence, we define the \textit{real dimension} of $M$ as $\textup{dim}_\RR(M)=\textup{dim}(E)\in\{-1,0,1,\dots,2n\}$. The case $\textup{dim}(E)=-1$ is devoted to the instance when $E=\emptyset$. 

Combining everything, the following theorem will be crucial in order to estimate the real dimension of the boundary zeroes of a polynomial that is zero-free in the unit ball.
\begin{theorem}{\cite[Theorem 10.5.6]{Rudin}}\label{theo:dimension}
    Let $\Omega\subset\RR^m$, $m\in\mathbb{N}$, be open and suppose that $\Phi:\Omega\to\SSS$ is $C^1$, non-singular and complex tangential. Then $\textup{dim}_\RR(V_x)=m\le n-1$, for all $x\in\Omega$.
\end{theorem}

\subsection{Semi-algebraic geometry}
We now turn to some tools from semi-algebraic geometry which we are going to need in order to successfully describe the boundary zero set of a polynomial. For more information, we refer the interested reader to the books \cite{Algebraic}, \cite{Loja}, and the article \cite{Denkowska}.

Fix $N\in\mathbb{N}$. A set $A\subset\RR^N$ is said to be \textit{semi-algebraic} if for any $x\in\mathbb{R}^N$, there exist a neighborhood $U=U(x)$ and a finite number of polynomials $f_i$, $g_{ij}$, $i=1,\dots,p$, $j=1,\dots,q$, $p,q\in\mathbb{N}$, such that
$$A\cap U=\bigcup\limits_{j=1}^{q}\bigcap\limits_{i=1}^{p}\{x\in U:f_i(x)=0, g_{ij}(x)>0\}.$$
Moreover, the set $A$ is said to be \textit{algebraic} if
$$A\cap U=\{x\in U: f_i(x)=0\}.$$

Given a polynomial $p\in\CC[z_1,\dots,z_n]$, the definition above dictates that $\mathcal{Z}(p)\cap\SSS$ is an algebraic set, as it is the intersection of the unit sphere $\SSS=\{z\in\CC^n: 1-||z||^2=0\}$ and the two algebraic sets $\{z\in\CC^n: \textup{Im}p(z)=0\}$, $\{z\in\CC^n: \textup{Re}p(z)=0\}$. Therefore, the set of the boundary zeroes of a polynomial may be regarded as an algebraic subset of $\RR^{2n}$.

To continue with, we are in need of certain notions about Nash submanifolds and Nash diffeomorphisms; see \cite[Chapter 2]{Algebraic}.
\begin{enumerate}
    \item[(i)] Let $A\subset\RR^M$ and $B\subset\RR^N$ be two semi-algebraic sets. A mapping $f:A\to B$ is characterized as \textit{semi-algebraic} if its graph is a semi-algebraic set of $\RR^{M+N}$.
    \item[(ii)] Let $A\subset\RR^N$ be semi-algebraic. A semi-algebraic function of class $C^\infty$ $f:A\to\RR$ is called a \textit{Nash function}. Moreover, given two semi-algebraic sets $A,B\subset\RR^N$, a semi-algebraic bijection of class $C^\infty$ $f:A\to B$ is called a \textit{Nash diffeomorphism}.
    \item[(iii)] A semi-algebraic subset $M$ of $\RR^N$ is said to be a \textit{Nash submanifold} of $\RR^N$ of dimension $d$ if for every point $x\in M$, there exists a Nash diffeomorphism $\phi$ from an open semi-algebraic neighborhood $\Omega$ of the origin in $\RR^N$ onto an open semi-algebraic neighborhood $\Omega'$ of $x$ in $\RR^N$ such that $\phi(0)=x$ and $\phi((\RR^d\times\{0\})\cap\Omega)=M\cap\Omega'$.
    \item[(iv)] Two Nash submanifolds $M$ and $N$ are considered to be \textit{Nash diffeomorphic} if there exists a bijection $f:M\to N$ such that both $f$ and $f^{-1}$ are Nash functions.
\end{enumerate}
All the definitions can be extended for any real closed field in the place of $\RR$. We choose to write everything in terms of $\RR$ to better suit our current work.

Before proving our first main result, we need a very useful proposition with regard to Nash submanifolds. This last proposition will aid us to a large extend in the pursuit of describing the set of boundary zeroes of a polynomial.

\begin{proposition}{\cite[Proposition 2.9.10]{Algebraic}}\label{prop:union}
    Let $A\subset\RR^N$ be a semi-algebraic set. Then $A$ is the disjoint union of finitely many Nash submanifolds $M_i$, each Nash diffeomorphic to an open hypercube $(-1,1)^{\textup{dim}(M_i)}$.
\end{proposition}

\subsection{Proof of Theorem \ref{zeroset}}
Having all the tools and results of the previous two subsections in hand, we are now ready to explicitly study the nature of the boundary zeroes of a polynomial that is zero-free in the unit ball.

\begin{proof}[Proof of Theorem \ref{zeroset}]
    Let $p\in\CC[z_1,\dots,z_n]$ be a polynomial satisfying the assumptions. Applying Proposition \ref{prop:union} to the semi-algebraic with respect to $\RR$ set $\mathcal{Z}(p)\cap\SSS$, we see that $\mathcal{Z}(p)\cap\SSS$ can be written as the disjoint union of finitely many Nash sumbanifolds $M_i$, while each $M_i$ is a semi-algebraic set, as well. In particular, for each index $i$ there exists a Nash diffeomorphism $\phi_i:(-1,1)^{\textup{dim}(M_i)}\to M_i$. However, since the closed field we are working on is $\RR$, by \cite[Chapter 8]{Algebraic} we know that each $\phi_i$ is actually real analytic. Obviously, in case $\textup{dim}(M_i)=0,$ for all $i$, and thus $M_i$ is a point, then $\mathcal{Z}(p)\cap\SSS$ is a finite set. On the contrary, suppose that there exists some index $i$ so that $\textup{dim}(M_i)=1,2,\dots,2n-2$ (it cannot be equal to $2n-1$, otherwise $p$ would be the zero polynomial, something that contradicts the hypothesis). Furthermore, consider a $C^1$ curve $\gamma_i:[0,1]\to(-1,1)^{\textup{dim}(M_i)}$. Then, the mapping $\phi_i\circ\gamma_i:[0,1]\to\mathcal{Z}(p)\cap\SSS$ is $C^1$, while the set $\phi_i\circ\gamma_i([0,1])$ is a compact subset of the (Z)-set $\mathcal{Z}(p)\cap\SSS$. As a consequence, $\phi_i\circ\gamma_i([0,1])$ is also a (Z)-set. Equivalently, by Remark \ref{equivalence-Z-PI} $\phi_i\circ\gamma_i([0,1])$ is a (PI)-set. It follows that $\phi_i\circ\gamma_i$ is a complex tangential curve; see \cite[Theorems 10.5.4, 11.2.5]{Rudin}. Since the choice of the curve $\gamma_i$ was arbitrary, we can infer through Remark \ref{equivalence-complex-tangential} that $\phi_i$ is also complex tangential. Finally, as a diffeomorphism, $\phi_i$ is non-singular, as well. Combining everything, Theorem \ref{theo:dimension} dictates that $\textup{dim}(M_i)\le n-1$ and we have the desired result.
\end{proof}

\section{Non-cyclicity of polynomials}
\subsection{Potential theory}
The proofs of our last two main results necessitate the use of arguments concerning potential theory and more specifically Riesz $\alpha$-capacity in the unit ball.  We start with a brief introduction of the subject. For more information we refer to \cite{Cohn}, \cite[Chapter 2]{Primer} and \cite{Pestana}.

For $\zeta,\eta\in\SSS$, we define their \textit{anisotropic distance} $d(\zeta,\eta)$ in $\SSS$ through the formula
$$d(\zeta,\eta)=|1-\langle\zeta,\eta\rangle|^{\frac{1}{2}}.$$
An important property of the anisotropic distance is that it remains invariant under composition with unitary matrices. As a matter of fact, given a unitary matrix $U$, we have $d(\zeta,\eta)=d(U\cdot\zeta,U\cdot\eta)$, for all $\zeta,\eta\in\SSS$, where again $U\cdot\zeta$ and $U\cdot\eta$ can be thought of as matrix multiplications.

For $\alpha\in(0,n]$ consider the non-negative \textit{kernel} $K_\alpha:(0,+\infty)\to[0,+\infty)$ given by 
$$K_\alpha(t)=
\begin{cases}
    t^{\alpha-n}, \quad\quad\hspace{0.1cm}\text{for }\alpha<n\\
    \log\left(\frac{e}{t}\right), \quad\text{for }\alpha=n.
\end{cases}$$
Let $\mu$ be any Borel probability measure supported on some compact Borel subset $E$ of $\SSS$. Then, the \textit{Riesz} $\alpha$\textit{-energy} of $\mu$ with respect to the anisotropic distance is defined to be the integral
$$I_\alpha[\mu]:=\iint\limits_{\SSS}K_\alpha(|1-\langle\zeta,\eta\rangle|)d\mu(\zeta)d\mu(\eta).$$
In this way, we may define the \textit{Riesz} $\alpha$\textit{-capacity} of $E\subset\SSS$ with respect to the anisotropic distance as the infimum
$$\textup{cap}_\alpha(E):=\left(\inf\{I_\alpha[\mu]:\mu\in \mathcal{P}(E)\}\right)^{-1},$$
where $\mathcal{P}(E)$ denotes the set of all Borel probability measures supported on $E$. If $E$ is any Borel subset of $\SSS$, then the Riesz $\alpha$-capacity of $E$ may be defined as $\textup{cap}_\alpha(E)=\sup\{\textup{cap}_\alpha(K):K\subset E, K \textup{ compact}\}$.

From this definition, we understand that $\textup{cap}_\alpha(E)=0$ if and only if $I_\alpha[\mu]=+\infty$ for all $\mu\in\mathcal{P}(E)$. On the other hand, if we find a $\mu\in\mathcal{P}(E)$ such that $I_\alpha[\mu]<+\infty$, then $\textup{cap}_\alpha(E)>0$.

\subsection{Proof of Proposition \ref{proposition}}
Before proceeding to the proof of Theorem \ref{non-cyclicity}, we will first prove a crucial proposition. In particular, we extend ideas appearing in \cite[Theorem 4.8]{Perfekt}, \cite[Theorem 21]{Vavitsas1} and \cite[Lemma 15]{Vavitsas2}; we show that the positivity of the Riesz $\alpha$-capacity of a submanifold of the unit sphere is invariant under complex tangential Nash diffeomorphisms. Through this next proposition, Theorem \ref{non-cyclicity} will follow effortlessly.

From now on, for $\theta=(\theta_1,\theta_2,\dots,\theta_m)\in\RR^m$ and $\epsilon>0,$ we denote $\Delta(\theta,\epsilon)=(\theta_1-\epsilon,\theta_1+\epsilon)\times(\theta_2-\epsilon,\theta_2+\epsilon)\times\dots\times(\theta_m-\epsilon,\theta_m+\epsilon).$ Also, $\Delta(0,\epsilon)$ is defined with respect to the origin $0\in \RR^m.$

Moreover, to simplify the notation, all different positive constants that appear in estimations below and that do note depend on $\theta,\theta',$ will be denoted by $C > 0.$

\begin{proof}[Proof of Proposition \ref{proposition}]
    First of all, following a similar argument as in the proof of Theorem \ref{zeroset}, $\Phi$ and $\Psi$ are actually real analytic diffeomorphisms. Given $\alpha\in(0,n]$, assume that $\textup{cap}_\alpha(M)>0$. Fix $\delta\in(-1,1)$. Then, the set $M':=\Phi(\Delta(0,\delta))$ is relative open in $M$ and as a result $\textup{cap}_\alpha(M')>0$, as well. Therefore, there exists a Borel probability measure $\mu$ supported on $M'$ such that $I_\alpha[\mu]<+\infty$. Consider $N':=\Psi(\Delta(0,\delta))$. Through $\mu$, pullback measures and pushforward measures, we may construct a Borel probability measure $\nu$ supported on $N'$. Indeed, for any measurable set $E$ in $N'$ define
    $$\nu(E)=\mu\left(\Phi|_{\Delta(0,\delta)}\circ\Psi^{-1}|_{N'}(E)\right).$$

    Our first aim in the proof is to show that there exists constant $C>0$ such that
    $$\frac{1}{C}||\theta-\theta'||^2\le|1-\langle\Psi(\theta),\Psi(\theta')\rangle|\le C||\theta-\theta'||^2,$$ 
    for all $\theta,\theta'$ in a proper open subset of $\Delta(0,1)$. This will eventually allow us to estimate $I_\alpha[\nu]$ and correlate it to $I_\alpha[\mu]$. 
    
    The left-hand side inequality is easier. As a matter of fact, since $\Psi$ is a diffeomorphism, we can find $C>0$ such that $||\Psi(\theta)-\Psi(\theta')||\ge C||\theta-\theta'||$, for all $\theta,\theta'\in\Delta(0,1)$. Then, we may observe that
    \begin{eqnarray}\label{lefthand}
     \notag   |1-\langle\Psi(\theta),\Psi(\theta')\rangle|&\ge&\textup{Re}(1-\langle\Psi(\theta),\Psi(\theta')\rangle)\\
     \notag   &=&\frac{||\Psi(\theta)-\Psi(\theta')||^2}{2}\\
        &\ge&C||\theta-\theta'||^2.
    \end{eqnarray}
    
    Nevertheless, the right-hand side inequality about the upper bound is more involved and requires careful handling. Suppose that $\Psi(\theta):=(\Psi_1(\theta),...,\Psi_n(\theta)),$ $\theta\in \Delta(0,1).$ Then there exists $r>0$ such that each function $\Psi_j$ has a power series expansion in $\Delta(0,r).$ Complexify each $\Phi_j$ and pick $\epsilon>0$ that does not depend on $\theta,\theta',$ so that
    $$\Psi_j(\theta):=\sum_{|k|=0}^{+\infty}a^j_k(\theta')(\theta-\theta')^k, \quad \theta'\in \Delta(0,\epsilon),\theta \in \Delta(\theta',\epsilon),$$
    $j=1,...,n.$ In addition, by the properly chosen $\epsilon$ and the complexification, for each $\theta' \in \Delta(0,\epsilon),$ the following estimation is true:
    $$|a_k^j(\theta')|=\Big|\frac{\partial^k\Psi_j(\theta')}{k!}\Big|\leq \frac{C}{\epsilon^{|k|}}.$$
    For a detailed background on arguments concerning real analytic and holomorphic functions of several variables we refer to \cite{Krantz,Laurent}.
    
    Now we turn to the quantity $|1-\langle\Psi(\theta),\Psi(\theta')\rangle|$. Since the anisotropic distance is invariant under compositions with unitary matrices, we may assume that $\theta'=0$ and $\Psi(0)=(1,0,\dots,0)$. Therefore, $|1-\langle\Psi(\theta),\Psi(0)\rangle|=|1-\Psi_1(\theta)|$ and the only thing left is to classify  and estimate the term $1-\Psi_1(\theta)$, for $\theta\in\Delta(0,\epsilon)$. In particular, we shall show that $|1-\Psi_1(\theta)|=O(||\theta||^2),$ $\theta$ close to the origin. Note that
    $$\Psi_1(\theta)=1+\sum\limits_{|k|=1}^{+\infty}a_k^1\theta^k, \quad\text{for all }\theta\in\Delta(0,\epsilon),$$
    where $a_k^1:=a_k^1(0)$ satisfy $|a_k^1|\le\frac{C}{\epsilon^{|k|}}.$ Furthemore, by our hypothesis $\Psi$ is complex tangential. Recall that complex tangentiality is defined through the same inner product we use in the anisotropic distance. As a consequence, the notion of complex tangentiality is invariant under compositions with unitary matrices, as well. So
    $$\langle J_\Psi(\theta)\cdot h,\Psi(\theta)\rangle=0, \quad\text{for all }h\in\RR^m,\theta\in\Delta(0,\epsilon).$$
    Applying this relation consecutively on the vectors $h_1=(1,0,\dots,0)$, $h_2=(0,1,\dots,0)$, $\dots$, $h_m=(0,0,\dots,1)$ and executing the necessary matrix multiplications along with certain straightforward calculations, we infer that $\frac{\partial\Psi_1(0)}{\partial\theta_j}=0$, for all $j=1,2,\dots,m$. This leads to $a_k^1=0$ whenever $|k|=1$ and hence
    $$\Psi_1(\theta)=1+\sum\limits_{|k|=2}^{+\infty}a_k^1\theta^k, \quad\text{for all }\theta\in\Delta(0,\epsilon).$$
    As a result
    \begin{eqnarray}\label{two sums}
      \notag |1-\langle\Psi(\theta),\Psi(0)\rangle|&=&|1-\Psi_1(\theta)|\\
     \notag  &\le&\sum\limits_{|k|=2}^{+\infty}|a_k^1||\theta|^k\\
       &=&\sum\limits_{|k|=2}|a_k^1||\theta|^k+\sum\limits_{|k|=3}^{+\infty}|a_k^1||\theta|^k.
    \end{eqnarray}
    Recall that $(|\theta_1|+|\theta_2|+\dots+|\theta_m|)^N=\sum_{|k|=N}\frac{N!}{k!}|\theta_1|^{k_1}\cdots |\theta_m|^{k_m}$, for $N\in\mathbb{N}.$ Moreover, $\frac{N!}{k!}\ge1$ in the sum. We will also make use of the known inequality $(\alpha+\beta)^p\le2^p(\alpha^p+\beta^p)$, where $\alpha,\beta\ge0$ and $p>0$. In fact, this last inequality can be used inductively for any finite number of summands. Combining these with (\ref{two sums}), we deduce that
    \begin{eqnarray*}
        |1-\langle\Psi(\theta),\Psi(0)\rangle|&\le&C\sum\limits_{|k|=2}\frac{2!}{k!}|\theta|^k+C\sum\limits_{N=3}^{+\infty}\epsilon^{-N}\sum\limits_{|k|=N}\frac{N!}{k!}|\theta|^k\\
        &\le&C||\theta||^2+C\sum\limits_{N=3}^{+\infty}\frac{(|\theta_1|+\dots+|\theta_m|)^N}{\epsilon^{N}}\\
        &\le&C||\theta||^2+C||\theta||^2\sum\limits_{N=3}^{+\infty}\frac{(|\theta_1|+\dots+|\theta_m|)^{N-2}}{\epsilon^N}.
    \end{eqnarray*}
    By potentially shrinking even more the domain where $\theta$ lies, we are allowed to assume that $|\theta_1|,|\theta_2|,\dots,|\theta_m|\le\frac{\epsilon^4}{m}$. In this way
    \begin{eqnarray*}
        \sum\limits_{N=3}^{+\infty}\frac{(|\theta_1|+|\theta_2|+\dots+|\theta_m|)^{N-2}}{\epsilon^N}&=&\sum\limits_{N=3}^{+\infty}\frac{\epsilon^{4N-8}}{\epsilon^N}\\
        &=&\sum\limits_{N=3}^{+\infty}\epsilon^{3N-8}<+\infty,
    \end{eqnarray*}
    since $\epsilon\in(0,1)$. All in all, we have proved that there exists proper $\delta>0$ such that
    $$|1-\langle\Psi(\theta),\Psi(0)\rangle|\le C||\theta||^2, \quad\text{for all }\theta\in\Delta(0,\delta).$$ 
    
    Returning to the general case, we have proved that
    $$|1-\langle\Psi(\theta),\Psi(\theta')\rangle|\le C||\theta-\theta'||^2, \quad\text{for all }\theta'\in\Delta(0,\delta)\text{ and }\theta\in\Delta(\theta',\delta).$$
    
    Obviously, identical arguments can be used to prove the exact same inequality for $\Phi$.
    
    We are finally ready for the last part of the proof which utilizes capacities. For the already defined measure $\nu$ on $N'=\Psi(\Delta(0,\delta))$, we have
    \begin{eqnarray*}
        I_\alpha[\nu]&=&\iint\limits_{N'}K_\alpha(|1-\langle\zeta,\eta\rangle|)d\nu(\zeta)d\nu(\eta)\\
        &=&\iint\limits_{\Delta(0,\delta)}K_\alpha(|1-\langle\Psi(\theta),\Psi(\theta')\rangle|)d\nu(\Psi(\theta))d\nu(\Psi(\theta'))\\
        &=&\iint\limits_{\Delta(0,\delta)}K_\alpha(|1-\langle\Psi(\theta),\Psi(\theta')\rangle|)d\mu(\Phi(\theta))d\mu(\Phi(\theta')).
    \end{eqnarray*}
    
    Trivially there exists $C>0$ so that $|1-\langle\Psi(\theta),\Psi(\theta')\rangle|\ge C,$ for all $\theta'\in \Delta(0,\delta), \theta\in\Delta(0,\delta)\setminus\Delta(\theta',\delta).$ Therefore,
    $$\int\limits_{\Delta(0,\delta)}\int\limits_{\Delta(0,\delta)\setminus\Delta(\theta',\delta)}K_\alpha(|1-\langle\Psi(\theta),\Psi(\theta')\rangle|)d\mu(\Phi(\theta))d\mu(\Phi(\theta'))<+\infty.$$
    Returning back, and via the positivity of the quantity we integrate, we obtain
    \begin{eqnarray*}
        I_\alpha[\nu]&\le&\int\limits_{\Delta(0,\delta)}\int\limits_{\Delta(\theta',\delta)\cap\Delta(0,\delta)}K_\alpha(C||\theta-\theta'||^2)d\mu(\Phi(\theta))d\mu(\Phi(\theta'))+C\\
        &\le&C\iint\limits_{\Delta(0,\delta)}K_\alpha(|1-\langle\Phi(\theta),\Phi(\theta')\rangle|)d\mu(\Phi(\theta))d\mu(\Phi(\theta'))+C\\
        &=&C\iint\limits_{M'}K_\alpha(|1-\langle\zeta,\eta\rangle|)d\mu(\zeta)d\mu(\eta)+C\\
        &=&CI_\alpha[\mu]+C<+\infty.
    \end{eqnarray*}
    Consequently, $I_\alpha[\nu]<+\infty$ and by extension $\textup{cap}_{\alpha}(N')>0$ which leads to $\textup{cap}_\alpha(N)>0$. Reversing the roles of $M$ and $N$, we get the desired equivalence.
\end{proof}

\subsection{Proof of Theorem \ref{non-cyclicity}}

\begin{proof}[Proof of Theorem \ref{non-cyclicity}]
    Let $\alpha>\frac{2n-(m-1)}{2}$. Then, as we said in the Introduction, the model polynomial $\pi_m$ is not cyclic. Moreover, by the explicit form of the set $\mathcal{Z}(\pi_m)\cap\SSS$, we may find a complex tangential Nash diffeomorphism $\Phi:(-1,1)^{m-1}\to\mathcal{Z}(\pi_m)\cap\SSS$, with $\mathcal{Z}(p)\cap\SSS$ being a Nash submanifold. In addition, by our hypothesis and Theorem \ref{zeroset}, there exists a Nash submanifold $N\subset\mathcal{Z}(p)\cap\SSS$ and a complex tangential Nash diffeomorphism $\Psi:(-1,1)^{m-1}\to N$. However, by \cite{Vavitsas2},  $\textup{cap}_\alpha(\mathcal{Z}(\pi_m)\cap\SSS)>0$ for $\alpha>\frac{2n-(m-1)}{2}.$  Therefore, all the necessary criteria are met and we can apply Proposition \ref{proposition} to get $\textup{cap}_\alpha(N)>0$. Immediately, this leads to $\textup{cap}_\alpha(\mathcal{Z}(p)\cap\SSS)>0$. By \cite[Theorem 16]{Vavitsas2}, we deduce that $p$ is not cyclic in $D_\alpha(\BB)$.
\end{proof}

\section*{Acknowledgements}
We would like to thank N. Chalmoukis, \L. Kosi\'{n}ski, T. Ransford and A. Sola for the helpful correspondence during the preparation of this work.

\bibliographystyle{plain}

\end{document}